\documentclass[a4paper,11pt]{article}
\usepackage[utf8]{inputenc}
\usepackage{amsfonts, amssymb, amsmath, amsthm}

\usepackage{srcltx}
\usepackage{color}
\usepackage{epsfig}
\usepackage{epsfig}
\usepackage[margin=1.25in]{geometry}
\usepackage{graphicx}
\usepackage{tikz-cd}

\newtheorem{lemma}{Lemma}
\newtheorem{definition}{Definition}
\newtheorem{claim}{Claim}
\newtheorem{corollary}{Corollary}
\newtheorem{theorem}{Theorem}

\newtheorem{proposition}{Proposition}
\newtheorem{remark}{Remark}

\newtheorem{question}{Question}
\newtheorem*{main-conjecture}{Main conjecture}
\newtheorem{example}{\bf Example}
 \def\BB{{\mathcal B}}

 \def\DD{{\mathbb D}}
 \def\ZZ{{\mathbb Z}}
 
 \def\RR{{\mathbb R}}
  \def\CC{{\mathbb C}}

 \def\la{\lambda}
 
  \def\ga{\gamma}
  \def\de{\delta}
  \def\La{\Lambda}
  
  \def\eps{\epsilon}

\title{Generalized hyperbolicity for linear operators }
\author{Patricia   Cirilo\,\,\,\,\,\,\,\, Bryce Gollobit\,\,\,\,\,\,\,\, Enrique  Pujals }

\begin{document}

\maketitle

\begin{abstract} It is introduced an open  class of linear operators on Banach and Hilbert spaces such that their non-wandering set is an infinite dimensional topologically mixing subspace. In certain cases, the non-wandering set coincides with the whole space.

\end{abstract}

\section{Introduction}
The present paper deals with  dynamics of  linear operators on infinite dimensional Banach spaces and how the dynamics changes for nearby linear operators. 

First, the notion of hyperbolicity is generalized (named generalized hyperbolicity) and it is shown that they are open in the norm topology. This class contains the classical hyperbolic ones and a  new one called shifted hyperbolic (which are non-hyperbolic and contains certain type of  weighted shifts).

After defining the ``bounded set" of a linear automorphism as the set of points in the space that it has infinitely many forward and backward iterates with bounded norm (that could depend on the initial point), it is proved  that  the bounded set of a generalized hyperbolic operator coincides with the non-wandering set and it  is a transitive set. Moreover,  the dynamic restricted to the bounded set is chaotic,  the shadowing property holds and periodic points are dense. Of course, in the hyperbolic setting the bounded set is reduced to the zero. However, this is not the case for shifted hyperbolic operator; their bounded set  is an infinitely dimensional subspace. In particular, it is also characterized when  the bounded set of  a  shifted hyperbolic   is dense in the whole  space.    In short, it is defined a class of linear operators that on one side form an open set in the space of linear dynamics with the norm topology and exhibiting rich dynamical behavior.

It follows from the results in \cite{FSW} that robustly transitive operators (acting on the full space) do not exist. However, for generalized  hyperbolic such that its bounded set is dense in the whole space (as it is the case of the strong shifted ones), it is possible to show that the bounded set of a nearby operator is ``large".

The present approach relies more on  dynamical systems techniques instead of classical spectral ones.

\section{Notations, definitions and main theorems}
\label{definitions}

Denote a (not necessarily separable) Banach space by $\BB$, the set of bounded linear operators from $\BB$ to itself by $L(\BB)$, and the set of linear automorphisms of $\BB$ by $L_{aut}(\BB)$.


\begin{definition} Let $T$ be a linear automorphism of  Banach space $\BB$. It is said that T is generalized  hyperbolic if  there exists a decomposition of $\BB$ in complementary closed subspaces, $\BB = E^-\oplus E^+$, such that
\begin{enumerate}
\item $T(E^+)\subset E^+$ and $T^{-1}(E^-)\subset E^-$;
    \item $T_{|E^+}$ and $T^{-1}_{|E^-}$ are uniform contractions. 
    
\end{enumerate}

\end{definition}

If the splitting is invariant by $T$, then $T$ is hyperbolic in the usual sense. We do not require that both subspaces are non-trivial, so both uniform contractions and expansions are generalized hyperbolic operators. Note that if the estimate $|T^n|_{E^+_x}|\leq C\lambda^n|x|$ holds for some $C \geq 1$ and $\lambda \in (0,1)$, then $\BB$ can be renormed so that $C = 1$. We will always assume this has been done. The definition is equivalent to requiring that the spectrum of $T|_{E^+}$ and $T^{-1}|_{E^-}$ are contained strictly inside the unit disc, by the spectral radius formula. 

The definition is motivated by the classic weighted shift and operator with that type of properties were studied in \cite{BCDMP} where it is proved that they have the shadowing property (the definition is given below).  In section \ref{sec.examples} are provided different types of examples of generalized hyperbolic operators.



First, we introduce a set that contains all the relevant ``bounded dynamics".

\begin{definition}
Given $T \in L_{aut}(\BB)$, the bounded set, denoted by $B(T)$, is the set of points $x$ for which there exist $K=K(x)>0$ and infinite sequences of positive integers $\{ k_n\}$ and $\{ m_n\}$ such that, 

\begin{enumerate}
\item $|T^{k_n}(x)|< K$ and 

\item $|T^{-m_n}(x)|< K.$  

\end{enumerate}
Moreover, let $B_N(T):=\{x: |T^n(x)|< N, \,\,\forall\, n\in \ZZ\}.$ 
The complement of $B(T)$, denoted as $UB(T)$,  is named the unbounded set of $T$ and  it can be written as the union of the sets $UB^+(T):=\{y: \lim_{n\to +\infty} |T^n(y)|=\infty \}$ and $UB^-(T)=\{ y: \lim_{n\to +\infty} |T^{-n}(y)|=\infty\}.$ 
\end{definition}

Recurrent points, points whose orbits accumulate on themselves, are contained in the bounded set. In general, $\overline{B(T)}$ is not a subspace.

\begin{definition} Given $T \in L( \BB )$ and an invariant set  $\La$ of  $T$, it is said that   $T_{|\La}$ is transitive if for any pair of open sets of $\La$, there is a forward iterate of one that intersects the other. It is said that $T_{|\La}$ is topologically mixing if for any pair of open sets in $\La$, any sufficiently large iterate of one set intersects the other. 
\end{definition}

Note that $T$ is transitive if and only if it has a dense orbit, provided $\BB$ is separable.

\begin{definition}
A point $x$ is called a non-wandering point if for any neighbourhood $U$ of $x$ there exists $k \geq 1$ such that $T^k(U)\cap U \not= \emptyset$. The set of non-wandering points for a map $T$ is denoted by $\Omega(T)$.
\end{definition}

Observe that the non-wandering set is not necessarily a subspace.

Whenever it is considered numerical simulation, it becomes relevant to deal with ``pseudo-orbits" and when they  are ``traced by true orbits":

\begin{definition}
Given a metric space and an homeomorphism $T$, a sequence $\{x_n\}_{n \in \mathbb{Z}}$ is a $\delta$-pseudo-orbit if there is a $\delta > 0$ such that $d(x_{n+1}, T(x_n)) < \delta$
for all   $n \in\mathbb{Z}$. A point $y$ is said to $\varepsilon$-shadow the pseudo-orbit $\{x_n\}$ if for all $n \in \mathbb{Z}$ 
$d(x_n, T^n(y))< \varepsilon$. 
\end{definition}

One can wonder about   systems with the 
property  in which numerical simulation does not introduce unexpected
behavior, in the sense that simulated orbits actually ‘follow’ real orbits.

\begin{definition}
Given a metric space and a homeomorphism $T$, it is said that $T$  have the shadowing property if given $\varepsilon > 0$, there is a $\delta > 0$ such that every $\delta$-pseudo-orbit is $\varepsilon$-shadowed by a point in $X$. If there is a constant $K > 0$ so that the $\delta$ and $\varepsilon$ can be chosen so that $\varepsilon < K \delta$, then $f$ is said to have the Lipshitz shadowing property. 
\end{definition}

Next theorem, summarize the most relevant dynamical properties of a generalized hyperbolic operator.

\begin{theorem}\label{bounded-transitive} Let $T$ be a  generalized hyperbolic operator, then:

\begin{enumerate}
\item \label{item tp}  $T_{|\overline{B(T)}}$ is topologically mixing;   

\item \label{item pp} the periodic points are dense in $B(T)$, $\Omega(T)=\overline{B(T)}$  and $\overline{B(T)}$ is a subspace;

\item \label{item BN} $\overline{B(T)}=\overline{\cup_{N>0} B_N(T)}$ and for any $N>0$ there exists $N'>N$ and $x\in B_{N'}(T)$ such the closure of its  orbit contains  $B_N(T)$;



\item \label{item sh} $T$ satisfies the Lipschitz shadowing property;

\item \label{item ub} the set $UB^+(T)\cap UB^-(T)$ is dense in $\BB$.

\end{enumerate}

\end{theorem}  
The proof is provided in section \ref{properties}. The fourth item of previous theorem is already proved in \cite{BCDMP}; here, we provided in section \ref{shadowing} a different proof based on the canonical one done for (non-linear) hyperbolic dynamics. Linear automorphisms with the unique shadowing are also classified (similar result was simultaneously proved in \cite{BM}).

In section \ref{sec.examples}, we provide examples of generalized hyperbolic operators such that $B(T)$ is dense in $\BB$. In view of the previous theorem, we call these transitive generalized hyperbolic operators. A subset of generalized hyperbolic operators disjoint from the hyperbolic ones is defined: 

\begin{definition} Let $T$ be a generalized  hyperbolic linear operator. $T$ is said to be a shifted hyperbolic operator if $T^{-1}(E^+) \cap E^-$ is a non-trivial subspace. 
\end{definition}

 In other words, $T$ is a  shifted hyperbolic operator if there exists a vector $v\in E^-$ such that it forward iterate is in $E^+$. Examples and constructions of shifted hyperbolic operators are presented in section \ref{sec.examples}. Generalized hyperbolic operators are comprised of shifted hyperbolic and hyperbolic operators only. 

 \begin{theorem} \label{thm decomposition}  Let $T$ be a generalized hyperbolic operator then:
 
 \begin{enumerate}
     \item Either it is hyperbolic or shifted hyperbolic;

     \item If there is a point in $E^-$ whose forward orbit has no component in $E^+$, then $T$ has a hyperbolic subspace. 

\end{enumerate}
     

\end{theorem}

Shifted hyperbolic operators  could still exhibit hyperbolic invariant subspace; see example 7. Moreover, in that example, there is a hyperbolic component in the closure of $B(T)$ (which in the example coincides with $\BB$).  In other words, second item of previous theorem could hold for some shifted hyperbolic operators.

\begin{definition}\label{d.transition} Let $T$ be a shifted hyperbolic operator; the subspace $E_0=T^{-1}(E^+)\cap E^-$ is called the transition subspace.  We denote by $\Sigma$, the shifted subspace, the smallest closed $T$-invariant subspace containing $E_0$.

\end{definition}

Note that the smallest closed $T$-invariant subspace containing $A\subset \BB$ is the closure of polynomials, in $T$ and $T^{-1}$, applied to points in $A$. For a hyperbolic operator, $B(T) = 0$, in contrast to the shifted case. 
 
 \begin{theorem}\label{thm SH} If $T$ is a shifted hyperbolic operator then $\overline{B(T)}$ is an   infinite dimension subspace that coincides with the shifted subspace. 
 
\end{theorem}

Shifted hyperbolic operators have periodic points of every period (theorem \ref{bounded-transitive} implies that they are dense in the shifted subspace), which can be described in terms of the transition vectors. This is done in proposition \ref{prop per points}.
 
Not all shifted hyperbolic operators are transitive. A direct sum of a hyperbolic map and a transitive generalized hyperbolic operator is clearly not. It may not be the case that a closed subspace where $T$ acts hyperbolically is not complemented, which leads us to the following definition.


\begin{definition}

Given $T \in L(\BB)$, it is said that $T$ has a hyperbolic component if there exists a closed invariant  subspace $F$ such that the induced map on $\BB/F$ is hyperbolic.
\end{definition}

Observe that having a hyperbolic component does not imply that there is a hyperbolic subspace as it is shown in example 8 in section \ref{sec.examples}. We have the following dichotomy:

\begin{theorem}\label{thm dichotomy} Let $T$ be a generalized hyperbolic operator. Either $\overline{B(T)} = \BB$, i.e. $T$ is transitive on $\BB$, or $T$ has a hyperbolic component. More precisely, if $\BB \neq \overline{B(T)}$, then the induced map on $\BB/{\overline{B(T)}}$ is hyperbolic. 

\end{theorem}

From   theorems \ref{thm SH} and \ref{thm decomposition}, it follows that if there is no hyperbolic components then  the shifted subspace coincide with the whole space. A classical type of examples of transitive shifted hyperbolic operators are the classical weighted shifts acting on $l^2(\ZZ)$ (see examples 1, 2 and 3 in section \ref{sec.examples}). In that case, any vector in $\BB$ can be written as an infinite sum of the iterates of the transition vector.  


With these examples in mind, we introduce the next definition:

\begin{definition}
Let $T$ be an operator acting on  a Banach space.
It is said that $T$ is strong shifted hyperbolic if 

\begin{enumerate}
    \item[--]$T$ is a transitive generalized hyperbolic operator,
    
    \item[--] There is a set of infinite countable independent subspace $(\hat E_k)$ such  that every $x \in \BB$ can be written as a convergent series $x = \sum_{k \in \ZZ} x_k$, where $x_k \in \hat E_k$ and each subspace $\hat E_k$ is the direct sum of finite iterates of the transition subspace.
    
\end{enumerate}

    
    



\end{definition}

In other words, there is an infinite  partition of $\mathbb{Z} = \bigsqcup_k P_k$ into finite subsets so that every $\hat E_k$ is generated by $(T^i(E_0))_{i\in P_k}.$



    
    



The next theorem shows that the absence of hyperbolic components and hyperbolic subspace characterized the strong shifted hyperbolic operators.
 \begin{theorem}
 \label{thm StSH} If $T$ is   generalized hyperbolic operator on a Hilbert space, and it does not contain either a hyperbolic  component or a hyperbolic invariant subspace,  then  $T$ is a strong shifted hyperbolic operator.

\end{theorem}

\begin{corollary}
If $T$ is  transitive generalized hyperbolic operator in a Hilbert space  without any invariant hyperbolic subspace, then $T$ is strongly shifted hyperbolic. 
\end{corollary}









One of the main property of a generalized hyperbolic is that it is an open property.

\begin{theorem}
\label{robustness} 

If $T$ is a generalized hyperbolic operator, there is $\eps>0$ such that for any $S$ such that $|T-S|<\eps$ then $S$ is generalized hyperbolic.

 Moreover, if $T$ is  shifted hyperbolic then $S$ is also  shifted hyperbolic.  In particular, if $T$ is shifted hyperbolic then for any operator nearby  holds that its bounded set is infinite dimensional.

\end{theorem}

It follows from the previous theorem and theorem \ref{bounded-transitive} that:
\begin{corollary}

Generalized  hyperbolic are robustly transitive on the non-wandering set. In particular, for shifted hyperbolic, the non-wandering set is an infinite dimensional robustly transitive subspace.

\end{corollary}

Observe that the theorem states that the splitting is robust. In the finite dimensional context, if a splitting is robust it holds that the splitting is dominated which is not the case of the splitting of a  shifted hyperbolic operator.

After previous theorem and knowing that there are generalized hyperbolic operators such that $\overline {B(T)}=\BB$ it is natural to wonder what about the bounded set $B(S)$ of a nearby operator $S.$
As it was mentioned in the introduction,there are no robustly transitive linear operators; so it follows that there are nearby systems for which the bounded set can not be dense in the whole Banach space (otherwise, it would be transitive and so the initial system would be robustly transitive). However, the following theorem shows that for nearby systems the bounded set is large. 

\begin{theorem} \label{large B} Let $T$ be a generalized hyperbolic operator. Then, for any $N>0$ and  $\epsilon$ there exists $\delta$ such that if $|S-T|<\delta$ then $B_N(S)$ is $\epsilon-$dense in $B_N(T)$.


\end{theorem}

    In  section \ref{sec.spaces}, we discuss which Banach spaces admit shifted hyperbolic operators, and provide a slightly more general construction than the weighted shifts. 

\begin{theorem} \label{Support} 
If $E_1$ and $E_2$ are Banach spaces with bases admitting bounded right shifts, then $\BB = E_1 \oplus E_2$ supports a transitive shifted hyperbolic operator. 
\end{theorem}

\section{Examples and relations with PDE}
\label{sec.examples}

We recall some classical examples of weighted shift, some of which are generalized hyperbolic operators. Moreover, in certain cases they are strong shifted hyperbolic.

From the theorem in the present paper, each of these examples provided an open class of examples.

\paragraph{1- Classical weighted shifts over $\ZZ$.} Let $\cal B$ either $l^p(\ZZ)$ or $C_0(\ZZ)$ and let $T$ defined on the canonical base as  $T(e_n)=a_n.e_{n+1}$ with $\lambda^{-1}< a_n< \sigma$ for $n<0$ and $\sigma^{-1}< a_n< \lambda$ for $n>0$ where $\lambda<1<\sigma.$ Observe that $T$ is strong shifted hyperbolic and the subspace decomposition is given by the subspaces $E^-=\{v\in {\cal B}: v_n=0\, \mbox{ for}\, n>0\}$
and  $E^+=\{v\in {\cal B}: v_n=0\, \mbox{ for}\, \leq 0\}$.

\paragraph{2- Product of generalized hyperbolic operators.} From those  examples, we can consider linear operators operators on large dimensional lattices. In fact, one can  consider  $T_1\dots T_k$ transitive  weighted shift operators acting on $l^p(\ZZ)$ or $C_0(\ZZ)$ and so 
$\bar T=T_1\times\dots\times T_k$ is a generalized hyperbolic. Since each $T_i$ is topologically mixing then $\bar T$ is transitive.

\paragraph{3- Operators in $L^2(\mathbb{R})$.} 
\label{L2}

Let $\psi:\mathbb{R}\rightarrow \mathbb{R}$ in $L^2(\mathbb{R})$  and let 
$$
\left[T_{t_0}\psi\right](x)=\lambda_{t_0}(x)\psi(x-t_0)
$$

where 

$$\lambda_{t_0}(x)= \exp(\int_0^{t_0} \gamma(x-s) ds)$$

with $\gamma(x)>\gamma_0>0$ for $x\leq 0$ and $\gamma(x)<-\gamma_0< 0$ for $x>0$. 

Observing that $L^2(\mathbb{R})=E^-+E^+$, where $E^-=\{\psi: \psi(x)=0 \ for \ x<0\}$ and $E^+=\{\psi: \psi(x)=0 \ for \ x>0\}$, it follows that  $T_{t_0}$ for $t_0>0$  is strongly shifted hyperbolic for that decomposition.  From theorem \ref{thm StSH} it follows that they are transitive.

\paragraph{4- One-parameter group of  weighted shifts.}\label{group}

Consider the family of linear  operators defined above and  observe  that 
$$T_t\circ T_s= T_{t+s},\,\,\,\,\,\,\,\, T_0=Id,$$
therefore, $(T_t)_{\in \RR}$ is a one parameter of weighted shift. More precisely, ${\mathcal T}: \RR\to L(L^2(R)) $ with ${\mathcal T}(t)=T_t$ is an action of $\RR$ over the linear operators of $L^2(\RR).$

\paragraph{5- Generalized hyperbolicity and  PDE's.}

Let us consider the  one-parameter group of operators introduced above and let us take the time derivative at $s=0$:

\begin{eqnarray*}
\partial_s T_s(\psi)(x)|_{s=0}= & &\partial_s [\lambda_s(x)\psi(x-s)]|_{s=0}\\
&=& (\partial_s\lambda_s(x))|_{s=0}\cdot \psi(x-s)+\lambda_s(x)\partial_s\psi(x-s)|_{s=0}\\
&=& \gamma(x) \psi(x) - \partial_x\psi(x).
\end{eqnarray*}

So, one can get a   linear partial differential equation (pde), 
$$\partial_s \psi= -\partial_x \psi + \gamma \cdot \psi,$$
in other words, a pde induced by the linear operator $\psi\to -\partial_x \psi+\gamma \cdot \psi$ acting on the Sobolev space $W^{k,p}$
and observe that the solution are given by the one-parameter group described above which are a one-parameter  group of strong shifted hyperbolic operators. 
In  the  particular  case  of $W^{k,2}$,  the  solutions  are  acting  on  a  Hilbert  space  and therefore by theorem 4 the solution are given by a transitive one-parameter group of liner operators.

We want to stress, that one can consider now small linear perturbation of previous pde. More precisely, one can consider equations of the form
$$\partial_s \psi= L(\psi),$$
where $L $ is a linear operator close to $\psi\to -\partial_x \psi+\gamma \psi$. The solution are given by a one-parameter  group of linear  operator close to the one-parameter  group of solution given by the initial equation.

One can also  consider a discretization of a PDE on the line where the time derivative is replaced by $\psi_{t+1}-\psi_t$ and the space derivative by $\psi_t(n+1)-\psi_t(n)$. In that sense, one can consider
the following linear equation in difference in $\RR \times B(\ZZ)$
\begin{eqnarray*}\label{pde 1}\psi_{t+1}(n)-\psi_t(n)= \psi_t(n-1)-\psi_t(n)+ V(n).\psi_t(n-1),
\end{eqnarray*}
where $\psi: \RR \times \ZZ\to \RR$ and $V:\ZZ\to \RR$; that equation can be understood as a discretization of the PDE $\partial_t \psi= -\partial_x \psi+ V(x)\cdot \psi(x).$ The equation  (\ref{pde 1}) can be rewritten as 
\begin{eqnarray*}\label{pde 2}\psi_{t+1}(n)= \psi_t(n-1)(1+ V(n-1)),
\end{eqnarray*}
and if $ V(n) > \alpha> 0 $ for $n<0$ and  $-2<\beta < V(n)< 0$  for $ n>0 $ then the operator defined by  that   equation  is a weighted shift and in particular, a generalized hyperbolic operator.

\paragraph{6- General construction of shifted hyperbolic.}

We provide a general class of examples based on  hyperbolic operators and one that mixed the hyperbolic directions.

Let $H$ be a bounded invertible hyperbolic operator  with splitting 
$E = E^s \oplus E^u$ and let  $U: \BB \rightarrow \BB$ be a bounded invertible operator such that $||H_{/E^u}||^{-1}< ||U||<||H_{/E^s}||^{-1}$ and  $U^{-1}(E^u) \subset E^u$ and $U(E^s) \subset E^s$, $U(E^u) \cap E^s \neq \varnothing$. 
Observe now  that  $T:=US$ is a shifted hyperbolic operator: if $x \in E^s$, then 
$|Tx| = |UHx| \leq ||U||||H_{/E^s}|| \leq \lambda |x|$
If $x \in E^u$, then
$$|T^{-1}x| = |H^{-1}U^{-1}x| \leq |(H_{/E^u})^{-1}| |U^{-1}x| = \lambda |x|$$ 
since $U^{-1}x \in E^u$. Thus the splitting for $T$ is same as that of $H$.

\paragraph{7- Transitive shifted hyperbolic with a hyperbolic subspace.} A variation of the classical weighted shifts, allows to get a shifted hyperbolic operator with a hyperbolic subspace but still $B(T)$ is dense in $\BB:$ 
\begin{enumerate}
    \item[--] the subspace  $<e_{-1}, e_0>$ is a hyperbolic invariant subspace with $T^{-1}(e_{-1})=\frac{1}{2}e_{-1}$ and $T(e_{0})=\frac{1}{2}e_{0}$;
    
    \item[--]  the transition subspace is given by $e_{-2}$ with  $T(e_{-2})=e_1$,

  \item[--] $T^{-1}(e_{-2})=\frac{1}{2}(e_{-1}+\frac{1}{3} e_{-3}),$ $T(e_{1})=\frac{1}{2}(e_{0}+\frac{1}{2} e_{2});$
  
  \item[--] for any $n\leq -3$, $T^{-1}(e_n)=\frac{1}{2}\frac{n}{n+1}e_{n-1}$ and for $n\geq 2$, $T(e_n)=\frac{1}{2}\frac{n}{n+1} e_{n+1}$.

\end{enumerate}

  Observe that $E^-=\{v\in {\cal B}: v_n=0\, \mbox{ for}\, n>0\}$
and  $E^+=\{v\in {\cal B}: v_n=0\, \mbox{ for}\, \leq 0\}$. 
Defining $v_n= e_{-1}+ \frac{1}{n}e_n $ for $n\leq -3$ holds that $T^{-1}(v_n)=\frac{1}{2} v_{n-1}$, and if $v_n= e_{0}+ \frac{1}{n}e_n $ for $n\geq 2$, it holds that $T(v_n)= \frac{1}{2}v_{n+1}$, therefore, $\oplus_{k\leq 0}T^{-k}(E_0)$ is dense in $E^-$ and $\oplus_{k> 0}T^{k}(E_0)$ is dense in $E^+.$

\paragraph{8- Shifted hyperbolic operator with a hyperbolic component but without a invariant hyperbolic subspace} In the present example, it is shown a generalized hyperbolic operator exhibiting a hyperbolic component without a hyperbolic subspace in the complement of $\Sigma.$ Let $T$ be a weighted shift acting on $l^2(\ZZ)$ such that $T(e_n)=7.e_{n+1}$ if $n\leq -2$,$T(e_n)=\frac{1}{7}.e_{n+1}$  if $n\geq 1$, $T(e_{-1})=7.e_1,$ and $T(e_0)=w+2.e_0$ with $w=\sum_{n < 0} 2^n.e_n$.
In other words, $T$ in all the entries different than $0$ it looks like the weighted shift with the caveat that $e_{-1}$ is sent to $e_1$ and $e_0$ is mapped into a combination of $w$ and  $e_0.$ Observe that  the subspace $\Sigma$ generated by $e_{-1}$ is closed and $B(T)$ is contained in $\Sigma$. It also holds that  $|T^k(w)|\geq (7/2)^k$ and so $|(\Pi_0\circ T)(e_0)|\geq 2^k$ (whee $\Pi_0$ is the projection over the subspace generated by $e_0$) and so $T$ has a hyperbolic component since $T$ acting on $l^2(\ZZ)/\Sigma$ is isometric to
$\Pi_0\circ T.$ On the other hand, $\frac{|\Pi_0(T^k(e_0))|}{|\Pi_\Sigma(T^k(e_0))|}\to 0.$ In other words, even if for the quotient operator, the vector $e_0$ is invariant and expanding, for $T$ its forward iterates remain expanded but their slope with  $\Sigma$ goes to zero, preventing the existence of an invariant subspace for $T$ in the complement of $\Sigma$ (using that the complement of $\Sigma$ has dimension one, if there is an invariant subspace has to be the limit of the iterates of $e_0$).


\section{Dynamical consequences of  Generalized  Hyperbolicity: Proof of theorem \ref{bounded-transitive}}
\label{properties}

The proof of theorem \ref{bounded-transitive} is divided in different subsections. First  it is shown that the bounded set  is topologically mixing. Then, it is shown that the set of periodic points are dense and that the bounded set is a subspace.  Finally, the shadowing property is revisited. We recall that generalized hyperbolics have the shadowing property and prove that uniqueness of the shadowing orbit is equivalent to hyperbolicity. At the end, it is described the unbounded set.

 First observe that  the  decomposition of $\BB$ in complementary closed subspaces, induces a decomposition  at  any point $x$ in two affine subspaces 
$$E^-_x:= E^-+ x,\,\,\, E^+_x:= E^++x,$$
and observe that 
$T^{-1}(E^-_x)= T^{-1}(E^-)+ T^{-1}(x)\subset E^-+T^{-1}(x)=E^-_{T^{-1}(x)}$ and  $T(E^+_x)= T(E^+)+T(x)\subset E^++T(x)=E^+_{T(x)}.$ In that way, given a neighborhood $B_r(x)$ of size $r$ around $x$, one can define  $B^-_r(x)=:B_r(x)\cap E^-_x$ and  $B^+_r(x)=:B_r(x)\cap E^+_x$.
 
Using that decomposition, even the subspaces $E^-$ and $E^+$ are not invariant, it is proved that they resemble certain properties of unstable and stable subspaces when are forwardly and backwardly iterated. In short, it is proved that the the iterates of a ball restricted to the subspace $E^-$ contain ball of larger radius in the $E^-$ subspace.
\begin{lemma}\label{st-ust}
Let $B_r(x)$ be a neighborhood of size $r$ around $x$, then 
\begin{enumerate}
    \item $T^k(B^-_r(x))\supset B^-_{\la^{-k}\cdot r}(T^k(x))$ for $k\geq 0;$

    \item   $T^{-k}(B^+_r(x))\supset B^+_{\la^{-k}\cdot r}(T^{-k}(x))$ for $k\geq 0$.
\end{enumerate}

\end{lemma}

\begin{proof}

 It is enough observe that $T^{-1}(B^-_{\gamma}(T(x)))\subset B^-_{\lambda.\gamma}(x)$ which implies $B^-_{\gamma}(T(x))\subset T(B^-_{\lambda.\gamma}(x))$ and taking $\lambda.\gamma=r$ we get $T(B^-_{r}(x))\supset B^-_{\lambda^{-1}\cdot r}(T(x))$. The proof of  second item is similar.
\end{proof}

\subsection{ Transitivity and topologically mixing  of the Bounded set}

The proof of the first item of  theorem \ref{bounded-transitive} is given in two steps; first it is proved that it is transitive and using that it is  proved that it is topologically mixing. For the first part, it is used lemma \ref{st-ust} to  show  that  forward  and  backward  iterates  of  two  open  sets eventually intersect, provided that some iterates remain bounded.

\vskip 5pt

 \noindent{\em Proof of item \ref{item tp} of theorem \ref{bounded-transitive}:} First, we proceed to prove the transitivity. 
 Given $U$ and $V$ open sets such that $B(T)\cap U$ and $V\cap B(T)$ are non empty and let $x$ and $y$ be  two points in these respective  sets; let $K(x)$ and $K(y)$ be  the positive constants that verify the definition and $k_n$ and $m_n$ be the two subsequences of iterates  $||T^{k_n}(x)||< K(x)$ and  $||T^{-m_n}(y)||< K(y)$. Observe that $T^{k_n}(B^-_r(x))\supset B^-_{\la^{-k_n}.r}(T^{k_n}(x))$ and $T^{-m_n}(B^+_r(y))\supset B^+_{\la^{-m_n}.r}(T^{-m_n}(y))$. 

For $n$ sufficiently large, and from the fact that $dist(T^{k_n}(x), T^{-m_n}(y))< K(x)+K(y)+dist(x,y)$, then $B^+_{\la^{-m_n}.r}(T^{-m_n}(y))\cap B^-_{\la^{-k_n}.r}(T^{k_n}(x))\neq \emptyset$ and therefore   $$T^{k_n}(B^-_r(x))\cap T^{-m_n}(B^+_r(y))\neq \emptyset;$$
moreover, any point $z$ in that intersection also belongs to $B(T)$ since $dist(T^l(z), T^{l}(y))$ and  $dist(T^{-l}(z), T^{-l}(x))$ go to zero as $l\to +\infty.$

Now, we proceed to show that $T_{|B(T)}$ is topologically mixing: 
From the first item in lemma \ref{st-ust}, since   $0$  is  a fixed point, it follows  that 
\begin{eqnarray}\label{zero}
T^k(B^-_\eps(0))\supset T^{k-1}(B^-_\eps(0)).
\end{eqnarray}

Also observe that for any $\eps>0$ the set    $\cup_{k>0}T^k(B^-_\eps T(0))$ is dense. In fact that, given an open set $V$ in fact by the transitivity, there is $x\in B_\eps(0)$ and $n$ arbitrary large  such that $T^n(x)\in V$, so, taking $x^-=\Pi^-(x)$ and since $dist(T^n(x), T^n(x^-))<\la^n.\eps$ it  follows that there is $k$ such that $T^k(B^-_\eps(0))\cap V\neq \emptyset$ and from  (\ref{zero}) it holds that  
\begin{eqnarray}\label{int zero}
T^n(B^-_\eps(0))\cap V\neq \emptyset\,\,\,\,\,\forall\,\, n\geq k.
\end{eqnarray}

The same argument can be performed for backward iterates and for $B^+_\eps(x)=:B_\eps(x)\cap E^+_x$ concluding that  given an open set $U$ there is $k>0$ such that $T^{-k}(B^+_\eps(0))\cap U\neq \emptyset.$ Let $z\in T^{-m}(B^+_\eps(0))\cap U$ and $\eps'$ such that $B^-_{\eps'}(z)\subset U$ and so $T^m(B^-_{\eps'}(z))\cap B_\eps(0)\neq \emptyset$; since $T^m(B^-_{\eps'}(z))\supset B^-_{\la^{-m}\cdot\eps'}(T^m(z))$ it holds that provided $m$ large enough $\Pi^-(T^m(B^-_{\eps'}(z)))\supset \Pi^-(B^-_{\la^{-m}\cdot\eps'}(T^m(z)))\supset B^-_\eps(0)$; taking $U'\subset T^m(B^-_{\eps'}(z))$ such that $\Pi^-(U')=B^-_\eps(0)$ it follows that $dist(T^n(U'), T^n(B^-_\eps(0))\leq \la^n.\eps$ and so by (\ref{int zero}) it follows that there is $k$ such that $T^n(U')\cap V\neq \emptyset$ for any $n\geq k$ and so $T^{n+m}(U)\cap V\neq \emptyset$ for any $n$ large enough.

\qed



\subsection{$\overline{B(T)}$ is a subspace and the  periodic orbits are dense }

The classic proof of hyperbolic dynamics proves the density of periodic points.

\vskip 5pt

 \noindent{\em Proof of item \ref{item pp} of theorem \ref{bounded-transitive}:} Let $z\in B(T)$;   if $z$ is not fixed (otherwise there is nothing to prove), a small neighborhood $U$ of $z$ is disjoint with its first iterate, ($U$ and $T(U)$ are disjoint), since $T_{|B(T)}$ is transitive, there is a forward iterate of $T(U)$ that intersects $U$ therefore, there is a point $x$ in $U$ such that $T^n(x)$ is also in $U$ and so, close to $x$. Let $\ga_1<\ga_2$ such that for any $y\in B^-_{\ga_1}(x)$ it holds that $(E^+ + y)$ intersects $ B^-_{\ga_2}(T^n(x))$ (which is also unique). Observe that if $T^n(x)$ is sufficiently close to $x$ then $\gamma_1$ and $\gamma_2$ can be taken close to zero. Let $\Pi^+:B^-_{\ga_1}(x)\to B^-_{\ga_2}(T^n(x))$ defined as $\Pi^+(y)$ being the unique intersection point of  $(E^+ + y)$ with $ B^-_{\ga_2}(T^n(x))$. If $n$ is large enough, it follows that $T^{-n}(B^-_{\ga_2}(T^n(x)))\subset B^-_{\ga_1}(x)$ and $\Pi^+\circ T^{-n}: B^-_{\ga_2}(T^n(x))\to B^-_{\ga_2}(T^n(x))$ is a contraction, and therefore there is $p\in B^-_{\ga_2}(T^n(x))$ such that $\Pi^+\circ T^n(p)=p.$ Let $\hat p=(\Pi^+)^{-1}(p)$ and observe that $T^n(\hat p)=p$. Let $\de>0$ be close to zero and larger than the distance from $x$ to $T^n(x)$. It follows that $T^n(B_\de^+(\hat p))\subset B^+_{\la^n.\de}(p)$ and so it is contained in $B^+_\de(p)$; moreover, $T^n: B_\de^+(\hat p) \to B_\de^+(p)$ is a contraction and so it has a unique fixed point $q$. In particular,  it follows that $T^n(q)=q.$
 
 To prove that $\overline{B(T)}$ is a subspace, it is enough to note that if $p$ and $q$ are periodic points, then $p+q$ is also a periodic point.

\qed

\subsection{Shadowing } 
\label{shadowing}

It was essentially shown in \cite{BCDMP} that generalized hyperbolic operators have the shadowing property: ``Suppose that $\BB = M \oplus N$, where $M$ and
$N$ are closed subspaces of $\BB$ with $T(M ) \subset  M$ and $T^{-1} (N) \subset N$. If $\sigma(T|_M) \subset \DD$ and $\sigma(T^{-1}|_N) \subset \DD$, then $T$
has the shadowing property". The classic proof of the shadowing property is sketched below. The proof has become standard, see proposition 8.20 in \cite{S} for the full details.

\vskip 5pt

\noindent{\em Proof of item \ref{item sh} of theorem \ref{bounded-transitive}:}  Given $n>0$ and $1\leq k\leq n$ it is chosen a sequences $(z^n_{n-k})$ such that $z^n_{n-k}= E^+(T^{-1}(z^n_{n-(k-1)}))\cap E^-(T(x_{n-k})).$  From the fact that $E^-$ and $E^+$ are backwardly and forwardly contracted,respectively, it follows that the sequence $(z^n_0)$ converges to a point $z$ that its orbits shadows the pseudo-orbit.

\qed

Recall that the spectrum $\sigma (T)$ of $T \in L(\BB)$ is the set of all $\lambda \in \mathbb{C}$ for which $T- \lambda I$ is not an automorphism. We denote by $\sigma_a(T)$ the set of approximate eigenvalues of $T$, that is the set of $\lambda \in \mathbb{C}$ for which given $\varepsilon > 0$ there exists an $x_{\varepsilon} \in S_\BB$ satisfying $|Tx_{\varepsilon} - \lambda x_{\varepsilon}| <\varepsilon$.

The map $f$ is said to have the unique shadowing property if there exist $\varepsilon_0, \delta_0 > 0$ such that for all $0< \varepsilon \leq \varepsilon_0$ there exists $0 < \delta \leq \delta_0$ so that every $\delta$-pseudo-orbit is uniquely $\varepsilon$-shadowed by a point in $X$. It is well known that hyperbolic linear maps have the unique shadowing property.

A homeomorphism $f$ of a complete metric space is said to be expansive if there is a constant $\varepsilon(f) > 0$ such that if $d(f^nx, f^ny) \leq \varepsilon(f)$ for all $n\in \mathbb{Z}$, then $x = y$. The constant $\varepsilon(f)$ does depend on the metric, but its existence does not depend on the metric (up to strong equivalence of metrics.) It is shown in \cite{BCDMP} that a linear automorphism of a Banach space is expansive if and only if for each $x \in S_{\BB}$ there is an $n \in \mathbb{Z}$ (depending on $x$) so that $|T^nx| \geq 2$ if and only if $T$ has no non-zero bounded orbits. If there is an $n > 0$ so that either $|T^n(x)| \geq 2$ or $|T^{-n}(x)| \geq 2$ for all $x \in S_{\BB}$, then $T$ is said to be uniformly expansive. 

 The following proposition due to \cite{Pi} provides an alternative characterization of shadowing for linear automorphisms. 
 
 \begin{proposition} $T \in L_{aut}(\BB)$ has the shadowing property if and only if there exists a constant $K > 0$ such that for any bounded sequence $\{z_n\}$ there exists $\{y_n\}$ such that 
 $$\sup_{n}|y_n| \leq K \sup_n|z_n|$$ 
 and $y_{n+1} = Ty_n + z_n$ for all $n \in \mathbb{Z}$. 
 \end{proposition}
 
 The sequence condition may be rephrased in terms of the surjectivity of a certain linear operator on $l^{\infty}(\BB)$. This perspective is useful because it follows immediately that shadowing is an open property for linear automorphisms, which is not true for general homeomorphisms. 

\begin{definition}Let $\BB$ be a separable Banach space. Define $l^{\infty}(\BB)$ by  
$$l^{\infty}(\BB) = \{ \{x_n\} \in \BB^{\mathbb{Z}} \ | \ \sup_{n \in \mathbb{Z}}|x_n| < \infty \}$$
Note that $l^{\infty}(\BB)$ is Banach with $|\{x_n\}|_{\infty}:=\sup_{n \in \mathbb{Z}}|x_n|$. $T \in L(\BB)$ induces a map $T_{\infty}: l^{\infty}(\BB) \rightarrow l^{\infty}(\BB)$ defined coordinate-wise by  
$T_{\infty}(\{x_k\})_n = x_{n+1} - Tx_n.$ 
\end{definition}

 The map $T \mapsto T_{\infty}$ is non-linear, but $T_{\infty}$ is a bounded linear operator whose kernel is exactly the set of bounded $T$ orbits. It follows that $T_{\infty}$ is injective if and only if $T$ is expansive. For $S, T \in L(\BB)$ automorphisms, $((T_{\infty} - S_{\infty})\{x_k\})_n = Sx_n - Tx_n$, so $|T_{\infty} - S_{\infty}| \leq |T-S|$. The proof of the next proposition is only a slight adaptation of the proof of proposition $1$ that appears in \cite{BCDMP}.

\begin{proposition} $T \in L_{aut}(\BB)$ has the shadowing property if and only if $T_{\infty}$ is surjective.
\begin{proof}Suppose that $T$ has the shadowing property. Let $\delta > 0$ be the constant corresponding to $\varepsilon = 1$. Choose $\{z_n\} \in l^{\infty}(\BB)$ with $|\{z_n\}|_{\infty} \leq \delta$. Given $x_0 \in \BB$, define a $\delta$-pseudo-orbit by 
$$x_{n+1} = Tx_n + z_n$$ 
Then there exists $x \in \BB$ with $|T^nx - x_n|< 1$ for all $n \in \mathbb{Z}$. If $y_n = x_n - T^nx$, then 
$$T_{\infty}(\{y_n\}) = x_{n+1} - T^{n+1}x - T(x_n - T^nx) = x_{n+1} - Tx_n = z_n$$
since $|\{y_n\}|_{\infty} < 1$, $T_{\infty}$ is surjective. 
$\\ \indent$ Conversely, suppose that $T_{\infty}$ is surjective. Let $\{x_n\}$ be a $\delta$-pseudo-orbit. Then $\{x_{n+1} - Tx_n\} \in B(0, \delta)$. By hypothesis there is $\{y_n\}$ such that 
$$T_{\infty}(\{y_n\}) = y_{n+1} - Ty_n = x_{n+1} - Tx_n$$ 
which implies that $y_n - x_n = T^n(y_0 - x_0)$. For any $n \in \mathbb{Z}$,  $|x_n - T^n(x_0 - y_0)| = |y_n| \leq |\{y_k\}|_{\infty}$. By the open mapping theorem there is $ r > 0$ such that $B(0,r) \subset T_{\infty}(B(0,1))$. By linearity, $B(0, \delta) \subset T_{\infty}B(0, \frac{\delta}{r})$. This implies every $\delta$-pseudo-orbit is $\frac{\delta}{r}$-shadowed. 
\end{proof}
\end{proposition}

\begin{corollary}Shadowing and unique shadowing are open properties.
\end{corollary}
\begin{proof} Linear surjections and automorphisms are open.
\end{proof}

The next proposition is well known when $E$ is complemented in $\BB$. 

\begin{proposition}Suppose $T \in L_{aut}(\BB)$ has the shadowing property. If $E \subset \BB$ is a closed $T$-invariant subspace, then the induced map $\hat{T}$ on the quotient space $\BB / E$ has the shadowing property. 
\end{proposition}
\begin{proof}Denote by $\hat{E}$ the quotient $\BB / E$ and by $p: \BB \rightarrow \hat{E}$ the quotient map. Then $\hat{T}$ is a linear factor of $T$ satisfying $\hat{T}p = pT$. If $p': l^{\infty}(E) \rightarrow l^{\infty}(\hat{E})$ is defined coordinate-wise by $p'(\{x_n\})_k = p(x_k)$, then $\hat{T}_{\infty}p' = p'T_{\infty}$. Since $T$ has shadowing $T_{\infty}$ is surjective, hence $\hat{T}_{\infty}$ is also surjective. 
\end{proof}

\begin{theorem}
\label{thm.shadowing}
Suppose that $T \in L_{aut}(\BB)$ has the shadowing property. Then the following are equivalent 
\begin{enumerate}
\item $T$ uniformly expansive
\item $T$ expansive
\item $T$ has the unique shadowing property
\item $T$ is hyperbolic

\end{enumerate}
\end{theorem}

\begin{proof}
$(1 \Leftrightarrow 2)$ Is proven in  \cite{BCDMP}. Without the assumption of shadowing, expansivity does not imply uniform expansivity. 

$(2 \Rightarrow 3)$ If $\{x_n\}$ is a $\delta$-pseudo-orbit shadowed by $y_1$ and $y_2$, then $|T^ny_1 - T^ny_2| = |T^n(y_1 - y_2)| \leq \varepsilon$ for all $n \in \mathbb{Z}$. Since $y_1 - y_2$ has bounded orbit, it follows that $y_1 - y_2 = 0$, so $T$  has unique shadowing. 

$(3 \Rightarrow 4)$ It is known that $T$ is uniformly expansive if and only if  $\sigma_a (T) \cap S^1 = \varnothing$ [H]. Since 
$\sigma (T) = \sigma_a(T) \cup \sigma_r (T)$ it remains to show that $\sigma_r(T) \cap S^1 \neq \varnothing$. By multiplying $T$ by a suitable $\lambda \in S^1$, since $\sigma_r(T) \subset \sigma_p(T^*)$, it suffices to show that $T^*$ doesn't have $1$ as eigenvalue. Suppose $T^*(f) = fT$ for $f \in E^*$. We show that $f=0$. Note that if $\{T^nx\}_{n \in \mathbb{Z}}$ has zero as an accumulation point, then $f$ vanishes along the orbit of $f$. Let $X$ be the closed subspace generated by such points. We may assume without loss of generality that $T^nx \rightarrow \infty$ as $n \rightarrow \infty$. Consider the bounded sequence $z_n = x$ when $n = 0$ and $z_n = 0$ otherwise. By proposition $1$, there exists $y_n$ such that 
$$y_{n+1} = Ty_n + z_n$$ 
and $|y_n| \leq K|x|$. We see that $y_n = T^ny_0 + T^nx$ when $n \geq 0$ and $y_n = T^ny_0$ when $n < 0$. Notice that any point $z$ with a forward or backward bounded orbit belongs to $X$ by uniform expansivity: without loss of generality suppose $|T^{-n}z| \rightarrow \infty$. If $N$ is chosen so that $|T^{-N}z| \geq 2 |z|$, it follows by induction that $|T^{kN}z| \leq 2^{-k}|z|$, and the right hand side tends to zero as $k \rightarrow \infty$. Thus $x+y_0 \in X$ and $y_0 \in X$. Since $X$ is a subspace $X = E$ and $f=0$. 
$\\ (4 \Rightarrow 1)$ Holds in the absence of shadowing. The converse does not. 
\end{proof}

\begin{remark}For a linear map of a finite dimensional space, the four conditions above are equivalent to the shadowing property.
\end{remark}

\begin{corollary} Any shifted hyperbolic linear operator   has the shadowing property, but not the unique shadowing property.
\end{corollary}

\begin{corollary} \label{bounded sequence} Suppose that $T$ has the shadowing property, and that $T$ is not hyperbolic. Then for every $\varepsilon > 0$, there is an $x$ such that 
$$1 - \varepsilon < |T^nx| < 1 + \varepsilon$$
for every $n \in \mathbb{Z}$. 
\end{corollary}
\begin{proof}This was used implicitly in [BCDMP]. We will use it several times in the sections that follow. Let $\lambda \in \sigma_a (T) \cap S^1$. Given $\delta > 0$, there is an $x \in X$ such that $|x| = 1$ and  $|Tx - \lambda x|< \delta.$ 
Define $\{y_n\}_{n \in \mathbb{Z}}$ by 
$y_n = \lambda^{|n|} x.$
Then $\{y_n\}$ is a $\delta$-pseudo-orbit since 
$|Ty_n - y_{n+1}| = |\lambda^n||Tx-\lambda x| = |Tx - \lambda x | < \delta$
and $|y_n| = 1$ since $\lambda \in S^1$ by assumption. 
\end{proof}

\subsection{Unbounded set of generalized hyperbolic operators}

Here we prove the last item of theorem \ref{bounded-transitive}.

\noindent{\em Proof of item \ref{item ub} of theorem \ref{bounded-transitive}}: The result is obvious for the case that $T$ is hyperbolic; in that case, since any vector in $E^-$ belong to $UB^+(T)$ and any vector in $E^+$ belongs to $UB^-(T)$ and so any non null vector belongs to $UB^+(T)\cap UB^-(T)$.

So, it remains to consider that $T$ is not hyperbolic, i.e., there is a vector $e\in E^-$ such that $T(e)\in E^+.$

To prove the result, it is enough to prove that for any point $z$ in $B(T)$ there exists a point $y$ in $UB^{\pm}(T)$ arbitrary close to $z$. Since the periodic points are dense, it is enough to assume that $z$ is periodic. Since $T$ is linear, it is enough to show that there is a vector $y$ with  norm one that $|T^n(y)|\to \infty$: in fact, given any $p$ periodic, the vector $p+\epsilon.y$ is close to $p$ and $|T^n(p+\epsilon.y)|\to \infty.$  

By corollary \ref{bounded sequence}, let $x$ such that $1-\eps<|T^n(x)|< 1+\eps$ for any $n$.  Let $(x_n)$ be the sequences defined inductively in the following way: $x_0=x$, $x_{n+1}=T(x_n)+\delta. \frac{T(x_n)}{|T(x_n)|}= T(x_n).(1+\frac{\delta}{|T(x_n)|})$ and observe that it is  $\delta-$pseudo-orbit. In particular,  it follows that
$$x_{n}=T^n(x_0)\cdot \displaystyle\Pi_{j=0}^{n-1}\left(1+\frac{\delta}{|T(x_j)|}\right).$$
It is claimed that $|x_n|\to \infty.$ If not, there is a subsequences $(x_{n_k})$ and $L>0$ such that $|x_{n_k}|< L$ and so $|T(x_{n_k})|<L.||T||$ and so there is $\gamma>1$ such that  $1+\frac{\delta}{|T(x_j)|}>\gamma.$ Given $m$, let $N_m$ the number of positive  integers $(n_k)$ smaller that $m$  such that $|x_{n_k}|< L$; clearly, $N_m\to\infty$ as $m\to \infty.$
Therefore, $$|x_m|> |T^m(x_0)| .\gamma^{N_m},$$ and since $1-\eps<|T^m(x)|$ it follows that $|x_m|\to \infty$, a contradiction.

 Let $y_0$ such that its  orbits $L.\delta-$shadows  the sequences $(x_n)$ so $y_0$ is close to $x$ and $|T^n(y)|\to \infty $ for $n\to \pm \infty.$

\qed




\section{Decomposition of generalized hyperbolic operators:\\ Proof of theorem \ref{thm decomposition}}
\label{sec decomposition}

In the present section it is shown that generalized hyperbolic operators are either hyperbolic or shifted hyperbolic. The proof follows showing that if no vector in $E^-$ goes to $E^+$ then the splitting is invariant and so the operator is hyperbolic. Then, it is shown that for a shifted  hyperbolic operator, if there are vectors  in $E^-$ such that its  forward iterates do not have a component in $E^+,$ then there is a hyperbolic subspace.

\vskip 5pt

\noindent{\em Proof of theorem \ref{thm decomposition}:} If both $E^+$ and $E^-$ are invariant, then $T$ is hyperbolic by definition. Suppose that $E^-$ is not totally invariant. Then there exists $x \in E^-$ such that $\Pi^+Tx \neq 0$. Write $Tx = \Pi^+Tx + \Pi^-Tx$. Then $x - T^{-1}\Pi^-Tx = T^{-1}\Pi^+Tx$. The vector on the left is in $E^-$ by backward invariance of $E^-$. Its image lies in $E^+$, so it is a transition vector. A similar argument shows that if $E^+$ is not invariant, then $T$ is shifted hyperbolic. Since shifted hyperbolic operators have non-zero bounded orbits, they cannot be hyperbolic (with respect to a possibly different splitting).

Suppose that $Orb(x) \subset E^-$. Then the smallest $T$-invariant subspace, $X$, generated by $Orb(x)$ is contained in $E^-$. We claim $T|_X$ is hyperbolic. If $p,q$ are polynomials, then 
$$|(p(T) - q(T^{-1}))x| = |T^{-1}T(p(T) - q(T^{-1})x| \leq \lambda|T(p(T) - q(T^{-1}))x| $$ 
so $T|_X$ is a uniform expansion on a dense set, hence on all of $X$. A similar argument holds for $E^-$.
\qed

\section{Dynamic of  shifted hyperbolic operators: Proof of theorem \ref{thm SH}}
\label{sec.shifted}

First, we prove theorem  \ref{thm SH}, then show that every periodic point of period $k$ can be represented as a sum along the orbit, under $T^k$, of a transition vector.

\vskip 5pt 

\noindent{\em Proof of theorem \ref{thm SH}:} Let  $v \in E^+\cap T(E^-)$, then the elements of $Orb(v)$ are linearly independent. Suppose that there are integers   $n_1 < n_2 < ... < n_k$ and non-zero scalars $a_1, ..., a_k \in \mathbb{C}$ such that 
$a_1T^{n_1}v + ... + a_kT^{n_k}v = 0$; then 
$v = b_2T^{j_1}v + ... + b_{k}T^{j_k}v$
for $b_i = -\frac{a_i}{a_1}$ and  $\;j_i = n_i - n_1$. Note that $j_i \geq 1$. Apply $T^{-1}$ to both sides. The right hand side is in $E^-$, while the left hand side is in $E^+$. Therefore $v = 0$. 


The second part of the theorem follows immediately from the fact that the orbit of the transition vector is in $B(T)$ and that $B(T)$ is a subspace.

\qed

\begin{proposition}\label{prop per points}Let $T$ be a generalized hyperbolic linear operator. Then $v_k = \sum_{n \in \mathbb{Z}}T^{nk}v$ is a periodic point of period $k$, for each $v \in E^+ \cap T(E^-) - \{0\}$. 
\end{proposition}
\begin{proof}
If $v \in E^+ \cap T^k(E^-)-\{0\}$, then given $n \in \ZZ$, 
$|T^{nk}v| \leq \lambda^{|n|k}|v| $ so
$\sum_{n \in \ZZ}|T^{nk}v| < \infty$
The sum converges absolutely, hence converges, since $\BB$ is Banach. We claim that $v_k$ has period $k$. This is not immediate, since the orbit of $v$ may not be a basic sequence, though it is clear that $T^k(v_k) = v_k$. 
\begin{claim} If $\sum_{n \in \mathbb{Z}}a_nT^nv = 0$, then $a_n =0$ for all $n$. 
\end{claim}
Suppose that $\sum_{n\in\mathbb{Z}}a_nT^nv = 0$. Then $\pi_+(\sum_{n\in\mathbb{Z}}a_nT^nv) = \sum_{n\geq 0}a_nT^nv = 0$ and $\pi_-(\sum_{n\in\mathbb{Z}}a_nT^nv) = \sum_{n<0}a_nT^nv = 0$. Suppose without loss of generality that $l \geq 0$ is the minimal $l$ so that $a_l \neq 0$. Applying $T^{-(l+1)}$, we have $-T^{-1}v = \sum_{n\geq l+2}a_nT^n$. The left hand side is in $E^-$, while the right hand side is in $E^+$, therefore $v = 0$.
$\newline \indent$ From the claim all the $v_k$ are non-zero. If $T^j(v_k) = v_k$, for some $0 < j < k$, then the difference is a sum of the above form, hence $v = 0$. 
\end{proof}

\begin{proposition}Let $T$ be a generalized hyperbolic linear operator. Then every periodic point $x$ with period $k$ has representation $x = \sum_{n \in \mathbb{Z}}T^{nk}v$, for some $v \in T^k(E^-) \cap E^+$. 
\end{proposition}
\begin{proof}Suppose $T^kx = x$. Writing $x = x_+ + x_-$ and applying $T^{-k}$,
\begin{equation}T^{-k}x_+ - x_+ = x_- - T^{-k}x_-
\end{equation}
It follows that $v = x_+ - T^kx_+ \in T^k(E^-) \cap E^+$, and that the sum $\sum_{n\in\mathbb{Z}}T^{nk}v$ converges absolutely. It suffices to show that $\sum_{n\geq0}T^{nk}v$ converges to $x_+$ and $\sum_{n < 0}T^{nk}v$ to $x_-$. The first series is telescoping with partial sums $x_+ - T^{-nk}v$, which converge to $x_+$. By $(3)$, the fist term of the second series can be rewritten as $x_- -T^{-k}x_-$, which shows that it also telescoping with partial sums $x_- - T^{-nk}x_-$, which converge to $x_-$. 
\end{proof}

\begin{corollary}If $E_0$ is separable, in particular if it is finite dimensional, then $\overline{B(T)}$ is separable. 
\end{corollary}

\section{Dichotomy: Proof of theorem \ref{thm dichotomy}}

First observe that any closed invariant subspace properly contained in $\overline{B(T)}$ can not be a hyperbolic component: if $L$ is a closed invariant subspace which is not $\overline{B(T)}$ then there is a non-trivial periodic point with orbit in  the complement of $L$ and that orbit projects onto a periodic one for the quotient map and therefore it can not be hyperbolic.  

Let $\hat E:= \BB/\overline{B(T)}$ be the quotient space and let $p: \BB\to \hat E$ be the quotient map.

\begin{claim}The quotient $\hat{E}:=\BB / \overline{B(T)}$ splits as a direct sum $\hat{E}=E_1 \oplus E_2$, where $E_1:=p(E^+)$ and $E_2:=p(E^-)$. 
\end{claim}

\begin{proof}
 
If $x \in Ker(p)=\overline{B(T)}$, let $x_n \in B(T)$ be such that $x_n \rightarrow x$. Write $x_n = \Pi^+x_n + \Pi^-x_n$ where $\Pi^+$ and $\Pi^-$ are as usual the projections over $E^-$ and $E^+$ respectively. Then $\Pi^+x_n$ and $\Pi^-x_n$ are in $B(T)$ since all forward (backward) iterates of $\Pi^+x_n \;(\Pi^-x_n)$ are bounded. It follows that $\Pi^+x$ and $\Pi^-x$ are in $\overline{B(T)}$, so $Ker(p) \subset Ker(p\Pi^+)$ and $Ker(p) \subset Ker(p\Pi^-)$. Then $p\Pi^+$ and $p\Pi^-$ factor through $p$, i.e. there are $\pi_1, \pi_2 \in L(\hat{E})$ such that $\pi_1p = p\Pi^+$ and $\pi_2p = p\Pi^-$. These maps are clearly projections.
\end{proof}

\begin{claim}$E_1$ and $E_2$ are an invariant splitting for the induced map on the quotient, $\hat{T}$.
\end{claim}
\begin{proof}

Given $x \in E^-$, write $Tx = \Pi^+Tx + \Pi^-Tx$. Applying $p$, we see that $\hat{T}px =p(Tx) = p(\Pi^-Tx) \in E_2$, by definition, since $\Pi^+Tx$ is the forward image of an element in the transition subspace. A similar argument shows that $E_1$ is invariant.

\end{proof}

The theorem follows since $p$ is an isometry provided that the quotients are non-trivial. This may also be seen less directly as a consequence of shadowing:

\begin{claim}If $E_1$ or $E_2$ are non-trivial, then $\hat{T}$ acts hyperbolically on them. 
\end{claim}

\begin{proof}
 
Applying proposition $3$ thrice, we see that $\hat{T}|_{E_1}$ and $\hat{T}|_{E_2}$ have the shadowing property. We claim that $\hat{T}|_{E_1}$ has the unique shadowing property, hence is hyperbolic by theorem $8$. If not, by corollary $4$ there is an $x \in E^-$ such that for all $n \in \mathbb{Z}$, 
$$\frac{1}{2} < ||pT^nx|| \leq |T^nx|$$ 
where $||.||$ denotes the quotient norm. This is absurd since $T^{-n}x\rightarrow 0$ as $n\rightarrow \infty$. A similar argument shows that $T|_{E_2}$ has the unique shadowing property (recall that $T$ has shadowing if and only if $T^{-1}$ does). Since the direct sum of hyperbolic maps is hyperbolic, the claim follows. 
\end{proof}

\begin{remark}If both $E_1$ and $E_2$ are trivial then $\overline{B(T)} = E$. It could be the case that only one is trivial. For instance, take the direct sum of a transitive generalized hyperbolic map with a uniform expansion or contraction.
\end{remark}

\section{ Strong shifted hyperbolicity: Proof of theorem \ref{thm StSH}}

The strategy of the proof consists in showing that any vector $w$ in the 
subspace generated by the backward iterates of the transition subspace can be written as a convergent infinite sum of vectors that belongs to the backward iterates of the transition subspace (similarly for the forward subspace); this is done in claim \ref{claim.closed}.

If $T$ is shifted hyperbolic, by theorem \ref{thm decomposition} it holds that    $E_0=E^-\cap T^{-1}(E^+)$ is a non trivial subspace and $T^{-1}(E_0)\cap E_0=0$ (since $T(E_0)\subset E^+$ and $E_0\subset E^-$).
Let $\Pi_n$ be the orthogonal projector from $\BB$ to $E_{n}:=T^{-n}(E_0)$ (observe that all the subspaces are isomorphic).

\begin{claim} The sequence of projections $(\Pi_n)$ does not contain a Cauchy subsequence.

\end{claim}

\begin{proof}
If it contains a Cauchy subsequence  $(\Pi_{n_k})$, there exists $\Pi$ such that $\Pi_{n_k}\to \Pi.$ Let $F=Ker(I-\Pi)$. It follows that $F\subset E^-$; it also holds that   $T^l(F)\subset E^-$ for any $l>0$: in fact, let $T^l\circ \Pi_{n_k}\circ T^{-l}$ and observe that the sequence  converges to $T^l\circ \Pi\circ T^{-l}$ and since the images of the projectors are   $T^l(E_{n_k})= E_{n_k-l}$ which are contained in $E^-$ and therefore since $E^-$ is closed it follows that the image of $T^l\circ \Pi\circ T^{-l}$. Now, one can consider the subspace generated by the iterates of $F$ and it follows that this subspace is invariant and contained in $E^-$ and therefore it is hyperbolic; a contradiction since $T$ has no invariant hyperbolic subspace.

\end{proof}

\begin{claim}
\label{no cauchy} Let $X$ be a separable metric space and let   $Z$ be a countable subset $X$ such that it does not contain a  Cauchy subsequence (in other words, any subsequence in $Z$ is not a Cauchy subsequence). Then, there is $c>0$ and an infinite countable number of  subsets of $Z$,     $Z_1, \dots, Z_n, \dots $ such that 
\begin{enumerate}
\item each $Z_j$ is finite;
\item for any $i\neq j$ if $z\in Z_j$ and $z'\in Z_i$ then $dist(z,z')\geq c.$
\end{enumerate}

\end{claim}

\begin{proof} Let us consider the sequences of positive numbers $(1/2^k)$ and for each $k$ let us consider a partition of $Z$ in subsets $Z^k_1, \dots, Z^k_n, \dots $ such that

\begin{enumerate}
\item diam $Z_j^k \leq 1/2^k$, that is, for any $z,z' \in Z^k_j$ then $dist(z,z')<1/2^k$;
\item for any $i\neq j$ if $z\in Z_j$ and $z'\in Z_i$ then $dist(z,z')\geq 1/2^k.$
\end{enumerate}

If for some $k$ holds that all the sets $(Z^k_i)_i$ are finite, then the claim is proved.

If not, if for any $k$ there is $Z^k_{n_k}$ which is infinite, then it can be chosen  an infinite sequences of nested sets $(Z^k_{n_k})$. Picking a point $z_k$ if any $Z^k_{n_k}$ it follows that $d(z_k, z_{k+1})< 1/2^k,$ then the sequences $(z_k)$ is a Cauchy sequences. A contradiction. 
\end{proof}

 Applying previous claim to the set of projectors $(\Pi_n)$, one get an infinite countable set $(P_k)$ of  subsets such that each $P_k$ is formed by a finite number of projectors.  Let us now consider for each $k$ the subspaces 
 $\hat E_k =  \oplus_{\Pi_i\in P_k} Ker(I-\Pi_i),$
 and let $\hat \Pi_k$ be the associated projector. By claim \ref{no cauchy} it holds that 
  there is  $c>0$,  such that for any $j, k$ \begin{equation}
 ||\hat \Pi_{k}-\hat \Pi_{j}||>c.
 \end{equation} 
Therefore,   there is $\beta<1$ such that for any unitarian vector it holds that
\begin{equation}
\label{bounded angle}
|< \hat \Pi_{k}(v),  \hat \Pi_{j}(v)>|<\beta.
\end{equation}

 In other words, we rearranged the subspaces $(E_n)$ in a collection of subspaces $(\hat E_k)$ in such a way that each $\hat E_k$ is only formed by a finite number of subspaces  and the ``slope" between the subspaces $(\hat E_k)$ is uniformly bounded by below. Also observe that since $\hat E_k$ is spanned by a finite union of subspaces of $(E_{-i})$ it follows that any vector in $\hat E_k $ has a forward iterate in $E^+.$ 

Now, it is consider a vector $v$ in the subspace generated by the subspace $(\hat E_k)$,  then  $v$ can be written  as a finite sum
$v=v_{j_1}+...+ v_{j_n}$ with $v_{j_i}= \hat \Pi_{j_i}(v).$

\begin{claim} It holds that there exists $C>0$ such that for any unitarian vector $v$ in the subspace generated by $(\hat E_k)_k$ it follows that  $|v_{j_i}|<C .$ 

\end{claim}

\begin{proof} 
Let us write $v$ as $\sum_k x_k \hat v_k$ where $\hat v_k$ is the normalized vector $\frac{v_k}{|v_k|}$; so 
$$|v|^2=\sum x_k^2+ \sum_{i\neq j} x_i x_j <\hat v_i, \hat v_j>.$$

By Cauchy-Schwartz, and the fact that  $|<\hat v_i, \hat v_k>|< |<\hat \Pi_i,\hat \Pi_k>| <\beta <1 $ it follows that 

 \begin{eqnarray*}
 |\sum x_i.x_j<\hat v_i, \hat v_k>| &<& \sum|x_i.x_j||<\hat v_i, \hat v_k>|\\& \leq& \beta [(\sum |x_i|^2) \cdot (\sum |x_j|^2)]^{1/2}= \beta \sum |x_i|^2. 
  \end{eqnarray*}
 In particular, 
$\sum x_i.x_j<\hat v_i, \hat v_k> \geq -\beta \sum |x_i|^2,$
therefore

 \begin{eqnarray*}
 |v|^2& =&\sum |x_k|^2+ \sum_{i\neq j} < \hat v_i, \hat v_j> x_i x_j \\
 &\geq &(1-\beta)\sum |x_k|^2- \beta \sum |x_k|^2 \\
 &=& (1-\beta)\sum |x_k|^2,
 \end{eqnarray*}
  therefore $\sum |v_k|^2$ is bounded by $\frac{1}{1-\beta}|v|^2$.
\end{proof}


\begin{claim}\label{claim.closed} Any vector in the closure of subspace generated by $(\hat E_k)_k$ can be written as convergent infinite sum of vectors in the sequence of subspace $(\hat E_k)_k$.

\end{claim}

\begin{proof}
 
Let  $v^n$ be a convergent sequences to a vector $w$;  then $v^n=\sum_k v^n_k$ with $v^n_k\in \hat E_k$ (observe that since $v^n$ is in the subspace generated by $(\hat E_k)_k$, for each $n$ only finite components are different than zero) and since the sequences is a Cauchy sequences and the norm of the vectors are compared to the sum of the square of the norm of the components, (i.e.: $|v^n-v^m|\approx \sum_k |v^n_k-v^m_k|^2$), it follows that  the sequences  $(v^n_k)_n$  are uniform Cauchy sequences and they converge to some vector $w_k$; since each  $\hat E_k$ is closed (there are a finite sum of closed subspaces) it follows that $w_k\in \hat E_k$ and so $w_k$ is the component of the vector $w$ in $\hat E_k$ and therefore $w=\sum w_k.$

\end{proof}

Recalling that each $\hat E_k$ is generated by only finite iterates of backward iterates of the transition subspace, then the theorem follows from previous claim.

\section{ Generalized hyperbolicity  is an open property: Proof of theorem \ref{robustness}}
\label{sec robustness}
We want to prove that for a an operator $S$ close to $T$ there exists a splitting $E^-_S\oplus E^+_S$ that satisfying the requirements of the definition of generalized hyperbolicity.

Let $E^-\oplus E^+$ be the subbundle decomposition for $T$ and since both  subspaces, $E^-$ and $E^+$, are closed and complementary in $\cal B$ it follows that for any vector $v\in \cal B$ there are unique vectors $v^-\in E^-, v^+\in E^+$ such that $v=v^-+v^+.$ Let $\Pi^\pm$ be the projection over $E^\pm $ respect to the splitting $E^-\oplus E^+$, i.e., $\Pi^\pm(v)= v^\pm.$  Provided $\alpha>0$ it is defined $C_\alpha(E^-)$, the cone of size $\alpha$ along the subspace $E^-$,  as the following set:

$$C_\alpha(E^-):=\{ v: \frac{|\Pi^+(v)|}{|\Pi^-(v)|} <\alpha\}.$$

\begin{remark}\label{contract} There exists $\lambda<1$ such that
$\frac{ |S(\Pi^+(v))|}{|S(\Pi^-(v))|} <\lambda.$ 
\end{remark}
In fact, that remark holds since vectors in $E^+$ are forwarded contracted, vector in $E^-$ are forwarded expanded and $S$ is close to $T$. Observe that the remark it does not state that the cone is forward invariant and also the statement is only about the iterates of the components in $E^-$ and $E^+.$

\vskip 5pt

\noindent{\em Proof of theorem \ref{robustness}:}
 Let $S=T+P$ be a linear map close to $T$. We choose $\alpha$ small, but larger than the norm of $P$, and construct a sequence of closed subspaces $M^0=E^-, M^1, \dots, M^n$ all of them complementary to $E^+$ and verifying:

\begin{enumerate}
\item $\Pi^-(M^j)=E^-$,
\item $M^j\subset S(M^{j-1}),$
  \item $M^j\subset C_\alpha(E^-)$,
  \item $S: S^{-1}(M^j)\to M^j$ is expanding (in the sense that the norm of the operator is larger than $1$). 
\end{enumerate}

That sequences is obtained by induction: assuming that  $\Pi^-(M^{j-1})=E^-$, let $M^{j-1}_{-1}$ be the subspace contained in $ M^{j-1}$ such that $\Pi^-(M^{j-1}_{-1})=T^{-1}(E^-)$ and we define $$M^j=S(M^{j-1}_{-1});$$ 
 observe that for $j=1$, $M^0_{-1}$ is $T^{-1}(E^-)$.



Let us prove first that $M^j\subset C_\alpha(E^-)$. Given $w\in M^{j-1}_{-1}$ it have to be shown that $S(w)\in C_\alpha(E^-)$, and for that one have to compute the norm of $\Pi^- \circ S(w)$ and $\Pi^+\circ S(w)$ and show that $\frac{|\Pi^+\circ S(w)|}{|\Pi^-\circ S(w)|} <\alpha.$ Let $w$ be an unitarian vector in  $ M^{j-1}_{-1}$; since $w\in C_\alpha(E^-)$ follows that $w= w^-_{-1}+w^+$ for some vector $w^+\in E^+$ and $w^-_{-1}\in T^{-1}(E^-)$ with $\frac{|w^+|}{|w^-_{-1}|}< \alpha.$ 


It follows that $S(w)= (T+P)(w)= T(w)+P(w)= T(w^-_{-1})+T(w^+)+ \Pi^-\circ P(w)+  \Pi^+\circ P(w)$. Observe that $T(w^-_{-1})\in E^-$, $|T(w^-_{-1})|\geq \la^{-1}.|w^-_{-1}|$  $T(w^+)\in E^+$, $|T(w^+)|\leq \la.|w^+|$ and $|\Pi^\pm\circ P(w)|<\epsilon. C.|w|,$ where $C$ is such that $||\Pi^{\pm}||< C.$  Then,
 \begin{eqnarray*}
 \label{inequality dominated}
 \frac{|\Pi^+(S(w))|}{|\Pi^-(S(w))|} \leq \frac{ \la|w^+|+C.\epsilon}{\la^{-1}|w^-_{-1}|-C.\epsilon}.
 \end{eqnarray*}
 
 Observe that for $\epsilon=0$ it follows that $\frac{ \la|w^+|}{\la^{-1}|w^-_{-1}|} < \la^{2}.\alpha$ so for  $\epsilon $ small and so $P$ close enough to zero it holds that 
$$\frac{|\Pi^+(S(w))|}{|\Pi^-(S(w))|}  < \la.\alpha .$$

 Let us prove now  that $\Pi^-(M^j)=E^-$ and $S: S^{-1}(M^j)=M^{j-1}_{-1}\to M^j$ is expanding:  observe that
 $ T\circ\Pi^-: M^{j-1}_{-1}\to E^-$ is surjective and  expanding (since $T\circ \Pi^-_{/M^{j-1}_{-1}}$ is close to $T: T^{-1}(E^-)\to E^-$), so, if $P$ is close to zero and $\alpha$ small, it follows that $\Pi^-\circ (T+P): M^{j-1}_{-1}\to E^-$ is also close to $T\circ \Pi^-$ and so it is  surjective and  expanding; hence, for $\alpha$ small follows that $\Pi$ is close to the identity and therefore  $S:S^{-1}(M^j)=M^{j-1}_{-1}\to M^j$  is expanding.

 Now, it is claimed  that the sequences $M^n$ ``has a limit". To precise that, observe that each $M^n$ are closed subspaces complementary to $E^+$ so for any vector $v$ there is a unique decomposition $v=v^n+v^+$ with $v^n\in M^n$ and $v^+\in E^+;$ so, one can define the projectors $\Pi^n$ as $\Pi^n(v)= v^n.$ It is going to be proved that the operators $\Pi^n$  converge to a linear operators $\Pi$ and it is defined $M$ as the kernel of is $I-\Pi$ and observe that in that case, that $M:=Ker(I-M)$ verifies that  $S^{-1}(E^-_S)\subset E^-_S$. In fact, since  $S^{-1}(M^n)\subset M^{n-1}$ it holds  that  $(I-\Pi^{n-1})\circ S^{-1}\circ \Pi^n=0$ and given that  $\Pi^n\to \Pi$ then $(I-\Pi)\circ S^{-1}\circ \Pi=0$ and so $S^{-1}(M)\subset M.$  The fact that $S^{-1}_{E^-_S}$ is a contraction  follows from the fact that $S^{-1}_{/M^n}$ is a contraction.

 To show the existences of the limit is enough to show that 
$$||\Pi^n-\Pi^{n-1}||< \lambda. ||\Pi^{n-1}-\Pi^{n-2}||.$$

Since $(M^{j}_{-1}\subset M^j$,  it holds that  $||\Pi^{n-1}_{|M^{n-1}_{-1}}- \Pi^{n-2}_{|M^{n-2}_{-1}},||\leq ||\Pi^{n-1}- \Pi^{n-2}||$ and  using remark \ref{contract} follows that

 \begin{eqnarray*}
  & & ||\Pi^{n}-\Pi^{n-1}||=\\
  & & ||S\circ \Pi^{n-1}_{|M^{n-1}_{-1}}- S\circ \Pi^{n-2}_{|M^{n-2}_{-1}}||< \lambda. ||\Pi^{n-1}_{|M^{n-1}_{-1}}-\Pi^{n-2}_{|M^{n-2}_{-1}}||= \lambda. ||\Pi^{n-1}-\Pi^{n-2}||.
 \end{eqnarray*}
The other two item in the thesis holds immediately since $S^{-1}(M^n)\subset M^{n-1}.$

\vskip 5pt

\qed

\section{Large bounded set for perturbation of a  generalized hyperbolic: Proof of theorem \ref{large B}}

To prove theorem \ref{large B} it is used that generalized hyperbolicity is an open property and that  they exhibit the  shadowing property. 



\vskip 5pt

{\em Proof of theorem \ref{large B}:} 
First observe that $S$ is generalized hyperbolic  and therefore $S$ also satisfies the shadowing property, in the sense that there exists $L$ such that any $C-$pseudo-orbit is $L.C-$shadowed.

Let us consider $B_N(T)$ and let $x$ be a vector in $B_N(T)$. Therefore, if $S$ $\delta'-$close to $T$ then, the orbit of $x$ is a $\delta.N-$pseudo orbit of $S$ and  so  it is $L.\delta.N-$shadowed by an orbit $\{S^j(y)\}_{j\in \ZZ}$ of $S$; in particular, for any $j$ holds that  $|S^j(y)|\leq (L.\delta+1).N$ and so for $\delta$ such that $\delta.L.N<\epsilon/2$ it holds that $B_{(1+\epsilon)N}(S)$ is $\epsilon/2-$dense in $B_N(T).$

\section{Spaces supporting shifted hyperbolic operators}
\label{sec.spaces}

We say $\BB$ supports a shifted hyperbolic operator if there is $T \in L_{aut}(\BB)$ such that $T$ is shifted hyperbolic. The splitting $\BB = E^+ \oplus E^-$ for a shifted hyperbolic operator requires that $E^+$ and $E^-$ be infinite dimensional. Indecomposable Banach spaces, ones for which $\BB = M \oplus N$ implies either $M$ or $N$ is finite dimensional, cannot support a shifted hyperbolic operator. There are Banach spaces for which every closed subspace is indecomposable.
$\newline \indent$
Recall that a basis is a sequence $\{e_n\}$ for which every $x \in \BB$ can be written uniquely as $x = \sum a_ne_n$. A basic sequence is a sequence that is a basis for its closed linear span. Every infinite dimensional Banach space contains an infinite basic sequence, but does not necessarily have a basis. If $\BB$ has a basis the shift $e_n \mapsto e_{n+1}$, extended linearly, need not be bounded, since the convergence of $\sum_{n\geq 1}a_ne_n$ doesn't imply the convergence of $\sum_{n\geq 1}a_ne_{n+1}$. Note that basic sequences do not need to be normalized. The following theorem is a concrete example of the general construction presented in example $6$.

\vskip 5pt

{\em Proof of theorem \ref{Support}:} Let $E_1$ and $E_2$ be Banach spaces which have bases $\{f_n\}_{n\geq 1}$ and $\{f'_n\}_{n \geq 1}$, respectively, that admit bounded right shifts $R_1$ and $R_2$. Note that 

$$e_{2n} = f_n, \; \; \mbox{and} \; \; e_{2n-1} = f'_n \; \; \; n \geq 1$$ 

is a basis of $\BB = E_1 \oplus E_2$, which we take to be equipped with max norm. Define $T: \BB \rightarrow \BB$ on basis elements: 
\begin{enumerate}
    \item[--]$T(e_{2n}) = \alpha e_{2n+2}, \; n \geq 0 $
    \item[--] $T(e_{2n-1}) = \beta e_{2n-3} \; n \geq 2 $
    \item[--] $T(e_1) = 2e_0$
\end{enumerate}
where $\alpha, \beta$ are chosen so $0 < \alpha < \frac{1}{2|R_1|}$ and $0 < \beta^{-1} < \frac{1}{2|L_2|}$, where $L_2: [f'_n]_{n \geq 2} \rightarrow [f'_n]_{n \geq 1}$ is the map $f'_n \mapsto f'_{n-1}$. Since $T$ is a continuous bijection it is an automorphism. $E_1 = E^+_T$ and $E_2 = E^-_T$ form a generalized hyperbolic splitting for $T$, with $e_1$ the transition vector.

\begin{example}Let $\{e_n\}_{n\geq 1}$ $c_0(\mathbb{N})$ be the canonical basis. The ``summing basis", $\{f_n\}$, where $f_n = e_1 + ... + e_n$ admits a bounded shift.
\end{example}
That $\{f_n\}$ is a basis, but not an unconditional basis (the sum does not converge for all permutations) is shown in \cite{AK}:  $\sum_{n\geq1}a_nf_n$ converges if and only if $a_n = b_{n+1} - b_n$ for some $\{b_n\} \in c_0$. The shift $S(f_n) = f_{n+1}$ is bounded, since if $x = \sum_{n\geq0}a_nf_n$, then
$$S(x) = \sum_{n\geq1}a_nf_{n+1} = \sum_{n\geq 1}a_n(f_n + e_{n+1}) = \sum_{n\geq 1}a_nf_n + \sum_{n\geq 1}a_ne_{n+1}$$

The first sum converges by hypothesis, and the second sum converges since $\{a_n\} \in c_0$. Since $c_0 \oplus c_0 \cong c_0$, the operator $T$ defined as in the proof of theorem $8$ is shifted hyperbolic. Note that it is not necessarily the case that $E^+ \cong E^-$. Using the notation of the previous theorem, take $E_1 = l^p(\mathbb{N})$ and $E_2 = l^{p'}(\mathbb{N})$ for $p \neq p'$, with right shifts on the canonical bases. Note that non-separable spaces, such as $l^{\infty}$, support shifted hyperbolic operators, but do not have bases.




\section{Questions}

In the present paper, it has been proved that under the hypothesis of generalized hyperbolicity, the non-wandering set is robustly  transitive. However, the non-wandering can change dramatically under perturbation (for instance, for a strong shifted hyperbolic its coincides with the whole space but for some perturbations, even  it remains infinitely dimensional, the non-wandering  `` becomes smaller"). In particular, that shows that the dynamics of two  nearby strong shifted hyperbolic operators can not be conjugated. However, one can wonder:

\begin{question} Are nearby strong shifted hyperbolic operators linearly semiconjugated?

\end{question}

In \cite{BM} it was proved that perturbations of some weighted shifts (the ones that are generalized hyperbolic in our context) are structural stable whenever are consider small Lipschitz bounded perturbations of them. Based on that, it is natural to wonder the following:

\begin{question} Are generalized hyperbolic linear operator  structural stable under the type of perturbations considered in \cite{BM}?

\end{question}

From \cite{FSW}, there dense perturbations of generalized hyperbolic such that their bounded set does not coincide with the whole Banach space:

\begin{question} Given a transitive hyperbolic operator, can the perturbations which remain transitive be classified?  

\end{question}

In view of theorem \ref{thm dichotomy}, previous question can be reformulated in terms of characterizing the perturbations that do not create a hyperbolic component. The result in \cite{FSW} proves that non-cyclic operators are dense and this is done through a perturbation techniques; more precisely,  the non-cyclical operators close that approximates, differs to the initial given operator only on a finite dimensional subspace; therefore, it should be considered ``infinite dimensional" perturbations.

In \cite{B} (see theorem 2.2) it is formulated a sufficient condition for transitivity, we wonder:

\begin{question} Does any generalized hyperbolic operator that its bounded sets span the Banach space satisfy the transitivity condition formulated by theorem 2.2 in \cite{B}?

\end{question}

On the other hand:

\begin{question} Do generalized hyperbolic operators admit a purely spectral characterization?

\end{question}
In view of theorem \ref{thm decomposition}, it would be enough to answer previous question for strong shifted hyperbolic.


\begin{question}If $\BB$ is a separable Banach space, does every transitive shifted hyperbolic operator factor through a hyperbolic one as in example $6$? 
\end{question}

\section{Further work beyond linear operators and other questions}

As it was mentioned before, one can consider non-linear bounded  perturbations of linear operators an in particular, in a more general context diffeomorphisms acting on a Banach  space endowed with the topology induced by the $C^r-$distance. For instance, two diffeomorphisms are $C^1-$ close if  given $\eps>0$, $f $ and $g$ are $C^1-$close if  for any $x\in \BB$,
\begin{itemize}
    \item[--] $|f(x)-g(x)|<\eps$,
    \item[--] $||D_xf-D_xg||<\eps.$
\end{itemize}

Observe that in this sense, two linear operators which are close in the topology of the norm are not $C^0$ close. In fact, the topology considered is closer to the one used in \cite{BM}.

For those type of diffeomorphisms, the notion is generalized linear operator can be extended.

\begin{definition}\label{dGWS def}

Let $f\in Diff(\BB).$  It is said that $f$ is a   generalized  hyperbolic diffeomorphisms if the derivative  has bounded   norm and for any $x\in \BB$ there exists a decomposition $\BB_x = E^-_x\oplus E^+_x$ such that
\begin{itemize}
\item $Df(E^+_x)\subset E^+_{f(x)}$ and $Df^{-1}(E^-_x)\subset E^-_{f^{-1}(x)}$;
    \item $Df_{|E^+_x}$ and $Df^{-1}_{|E^-_x}$ are uniform contractions.
\end{itemize}

\end{definition}

Observe that the definition is 
a natural extension of the generalized hyperbolic one, but in the present case the decomposition changes point by point and the semi-invariance now is provided by the derivative of the map. In particular, any generalized hyperbolic operator is a generalized hyperbolic diffeomorphisms.

A similar proof to the one done in theorem \ref{robustness}, shows that the class of generalized hyperbolic diffeomorphisms is an open class (and in view of previous paragraph, a non-empty class).

Similarly, the bounded set can defined and again, there are generalized hyperbolic diffeomorphisms that the closure of their bounded set coincides with the whole space. 

As in the context of classical theory of hyperbolic dynamical systems, the goal is to show that there exist stable and unstable manifolds for generalized hyperbolic diffeomorphisms in the sense that the stable is forward invariant, the unstable is backward invariant. With this construction in mind, it would follow that it is possible to get a similar result to \ref{bounded-transitive}. 

\begin{question} Are generalized hyperbolic diffeomorphisms stable?

\end{question}

\vskip 10pt

\hspace{-1.2cm}
\begin{tabular}{l l l l l}
\emph{Patricia Cirilo}
& &\hskip 50pt 
\emph{Bryce Gollobit}
& & \hskip 50pt 
\emph{Enrique Pujals}
\\

UNIFESP  && \hskip 30pt Graduate Center-CUNY
&&\hskip 30pt Graduate Center- CUNY

\end{tabular}

\end{document}